\documentclass[11pt]{amsart}
\usepackage{amsmath,amsfonts,amssymb,amscd,amstext,mathrsfs}
\usepackage[utf8]{inputenc}
\usepackage{hyperref}
\usepackage{verbatim}
\usepackage[up,bf]{caption}
\usepackage{graphicx}
\usepackage{float}

\usepackage{times}
\usepackage{enumerate}

\input xy
\xyoption{all}

\newtheorem{theorem}{Theorem}[section]
\newtheorem{proposition}[theorem]{Proposition}

\newtheorem{corollary}[theorem]{Corollary}

\theoremstyle{definition}
\newtheorem{definition}[theorem]{Definition}

\newtheorem{example}[theorem]{Example}

\newcommand{\RR}{\mathbb{R}}
\newcommand{\CC}{\mathbb{C}}
\newcommand{\DD}{\mathbb{D}}

\newcommand{\NN}{\mathbb{N}}
\newcommand{\ZZ}{\mathbb{Z}}

\newcommand{\cA}{\mathcal{A}}

\newcommand{\cO}{\mathcal{O}}

\newcommand{\Int}{\text{Int}\,}

\numberwithin{equation}{section}
\numberwithin{figure}{section}

\newcommand\N{\mathbb{N}}

\newcommand\R{\mathbb{R}}

\newcommand\igot{\mathfrak{i}}

\renewcommand\igot{\mathfrak{i}}

\newcommand\e{\mathrm{e}}

\renewcommand\imath{\igot}

\newcommand\Id{\mathrm{Id}}

\numberwithin{equation}{section}

\begin{document}
\title{Families of proper holomorphic maps}

\author{Barbara Drinovec Drnov\v sek and Jure Kali{\v s}nik}

\address{Barbara Drinovec Drnov\v sek, Faculty of Mathematics and Physics, University of Ljubljana, and Institute of Mathematics, Physics, and Mechanics, 
Jadranska 19, 1000 Ljubljana, Slovenia}
\email{barbara.drinovec@fmf.uni-lj.si}

\address{Jure Kali{\v s}nik, Faculty of Mathematics and Physics, University of Ljubljana, and Institute of Mathematics, Physics, and Mechanics, 
Jadranska 19, 1000 Ljubljana, Slovenia}
\email{jure kalisnik@fmf.uni-lj.si}


\date{\today}


\begin{abstract}
Given a smooth, open, oriented surface $X$ endowed with a family of complex structures $\{J_b\}_{b\in B}$ depending continuously 
on the parameter $b$ in  a metrisable space $B$, we construct a continuous family of proper holomorphic maps
$F_{b}:(X,J_b)\to\CC^{2}$, $b\in B$.
\end{abstract}

\maketitle

\centerline{\em Dedicated to Josip Globevnik}


%
%
%
%
\section{Introduction}\label{sec:intro}

Every smooth, open, oriented surface $X$ endowed with an almost complex structure $J$ is a Riemann surface.
Therefore, by choosing  a continuously varying family of almost complex structures $(J_b)_{b\in B}$ for some parameter space  $B$,
we determine a family of open Riemann surfaces $(X,J_b)_{b\in B}$.
In 2025, Forstneri\v c  \cite{Forstneric2025} initiated the study of continuous maps $F$ from $B\times X$ to Euclidean space, or more generally, to an Oka manifold,
such that for each ${b\in B}$ the map $F(b,\cdot)$ is holomorphic on the Riemann surface $(X,J_b)$.
In this framework, he obtained the Runge and Mergelyan approximation theorems, as well as the Weierstrass interpolation theorem. 

Our main result answers in part the question 
raised by Forstneri\v c \cite[Problem 8.7 (a)]{Forstneric2025} concerning the existence of proper holomorphic maps in this setting:

\begin{theorem}\label{Main theorem} Let $X$ be a smooth, connected, open, oriented surface, $B$ a metrisable space, and 
$\{J_b\}_{b\in B}$ a continuous family of complex structures on $X$ of class $C^{(k,\alpha)}$ with $k\in \mathbb Z_+,0<\alpha<1$. 
Then there exists a continuous map $F:B\times X\to\CC^{2}$ such that for every $b\in B$ the map $F(b,\cdot):(X,J_b)\to\CC^{2}$ is
proper holomorphic.
\end{theorem}

Precise definitions will be given in the next section.
It is classical that for every open Riemann surface there is a proper holomorphic immersion into $\CC^2$ and 
a proper holomorphic embedding into $\CC^3$, see \cite[Theorem 2.4.1]{Forstneric2017} and the references therein.

By increasing the dimension of the target Euclidean space by one, we obtain a family of proper holomorphic immersions:

\begin{corollary}
Let $X$ be a smooth, connected, open, oriented surface, $B$ a finite CW complex or a smooth manifold, and 
$\{J_b\}_{b\in B}$ a continuous family of complex structures on $X$ of class $C^{(k,\alpha)}$ with $k\ge 1,0<\alpha<1$. 
Then there exists a continuous map $G:B\times X\to\CC^{3}$ such that for every $b\in B$ the map $G(b,\cdot):(X,J_b)\to\CC^{3}$ is
a proper holomorphic immersion.
\end{corollary}

\begin{proof}
By  \cite[Corollary 8.3]{Forstneric2025} there exists a continuous function $h:B\times X\to\CC$ such that $h(b,\cdot):(X,J_b)\to\CC$
is a holomorphic immersion for every $b\in B$. Let  $F:B\times X\to\CC^{2}$ be a continuous map such that for every $b\in B$ the map $F(b,\cdot):(X,J_b)\to\CC^{2}$ is
proper holomorphic, provided by Theorem \ref{Main theorem}.
Then the map $(F,h):B\times X\to\CC^{3}$ is continuous and  for every $b\in B$ the map $(F,h)(b,\cdot):(X,J_b)\to\CC^{3}$ is
a  proper holomorphic immersion.
\end{proof}

We extend to families the result of Forstneri\v c and Globevnik \cite[Theorem 1.4]{FGMRL}, Alarc\'on and L\'opez \cite[Corollary 1.1]{ALJDG}, and Andrist and Wold \cite[Theorem 5.6]{AWAIF} on proper harmonic maps from open Riemann surfaces to $\R^{2}$:

\begin{theorem}\label{harmonic}
Let $X$ be a smooth, connected, open, oriented surface, $B$ a metrisable space, and 
$\{J_b\}_{b\in B}$ a continuous family of complex structures on $X$ of class $C^{(k,\alpha)}$ with $k\in \mathbb Z_+,0<\alpha<1$. 
There exists a continuous map $H:B\times X\to\R^{2}$ such that for every $b\in B$ 
the map $H(b,\cdot):(X,J_b)\to\R^{2}$ is proper harmonic.
\end{theorem}

The proof relies on the proof of Theorem \ref{Main theorem} and we postpone it to Section \ref{sec:proof}.

By Remmert's proper mapping theorem, the image of an analytic subvariety under a proper holomorphic map is an analytic subvariety.
Therefore, the following corollary provides, in particular, a path of complex analytic subvarieties in $\CC^2$ 
from the one parametrised by the complex line to the one parametrised by the unit disc.

\begin{corollary}
Let $X$ be a smooth, connected, open, oriented surface. Let $J_0, J_1$ be
complex structures on $X$ of class $C^{(k,\alpha)}$ with $k\in \mathbb Z_+,0<\alpha<1$. 
There exist a continuous family $\{J_b\}_{b\in[0,1]}$ of complex structures on $X$ of class $C^{(k,\alpha)}$ and
 a continuous map $F:[0,1]\times X\to\CC^{2}$ such that for every $b\in [0,1]$ the map $F(b,\cdot):(X,J_b)\to\CC^{2}$ is proper $J_b$-holomorphic.
\end{corollary}

\begin{proof}
Each complex structure determines a compatible Riemannian metric on $X$ of the same smoothness class.
 Convex combinations of these metrics yield a path
connecting the two, which in turn induces a corresponding path of almost complex structures on $X$ of the same smoothness class; 
see, for example \cite[Lemma 1.9.1]{AlarconForstnericLopez2021}.
Then the conclusion follows from Theorem \ref{Main theorem}.
\end{proof}

The main idea in the proof is constructing a convergent sequence of maps on an exhausting
sequence of Runge compact sets of $X$  in a way similar to constructions in \cite{ALJDG,AlarconForstneric2013IM,DDFJMAA}.
In  \cite{ALJDG}, Alarc\'on and L\'opez constructed a proper conformal minimal immersion
from any open Riemann surface $M$ into $\R^3$ with its image in a wedge, 
and in  \cite{AlarconForstneric2013IM}, Alarc\'on and Forstneri\v c  obtained a proper holomorphic immersion from any open Riemann surface $M$ into
$\CC^2$ directed by an Oka cone. 
The main tool in our construction is the Mergelyan approximation theorem for proper families of compact Runge sets recently proven by Forstneri\v c \cite{Forstneric2025}.
When the parameter space $B$ is not compact, one has to deal with nonconstant proper families of compact Runge subsets of $X$,
which are present already in the noncritical case, i.e., when the topology of $X$ is trivial.

\section{Preliminaries}\label{Prelim}

We use the notations $\NN=\{1,2,3,\ldots\}$ and $\ZZ_{+}=\{0,1,2,3,\ldots\}$ for the set of natural numbers,
respectively the set of nonnegative integers. If $K$ is a compact topological space and $f:K\to\CC$
is a continuous function, we denote by $\|f\|_{K}$ the supremum norm of $f$ on $K$.

Throughout the paper, we denote by $X$ a smooth, connected, open, oriented, Hausdorff, second countable surface. We are
interested in families of complex structures on $X$, parametrised by some topological space $B$ as defined in \cite{Forstneric2025}. 
A complex structure on $X$ is given by
a section $J\in\Gamma(\text{End}(TX))$ of the bundle of endomorphisms $\text{End}(TX)$ of the tangent bundle $TX$ of $X$
that satisfies the condition $J^{2}=-\Id$.  We always assume that $J$ induces on $X$ the given orientation of $X$.
Since the tangent bundle $TX$ 
is trivial, the bundle $\text{End}(TX)$ is isomorphic to the trivial bundle $X\times\text{End}(\RR^{2})$. If we choose a trivialisation
of $\text{End}(TX)$, we can identify sections of $\text{End}(TX)$ with functions from $X$ to $\text{End}(\RR^{2})$. If we furthermore
choose a Riemannian metric on $X$, we can define Banach spaces $\Gamma^{(k,\alpha)}(\text{End}(TX)|_{\Omega})$
of sections of $\text{End}(TX)$ of  H\" older class $C^{(k,\alpha)}(\Omega)$
 for any $k\in\ZZ_{+}$, $0<\alpha<1$ and any relatively compact domain $\Omega\subset X$.
A complex structure $J\in\Gamma(\text{End}(TX))$ is locally of class $C^{(k,\alpha)}$ if $J|_{\Omega}\in\Gamma^{(k,\alpha)}(\text{End}(TX)|_{\Omega})$
for every relatively compact domain $\Omega\subset X$.

\begin{definition}
Let $B$ be a topological space, $k\in\ZZ_{+}$ and $0<\alpha<1$. A \textbf{continuous
family of complex structures on $X$} of class $C^{(k,\alpha)}$, parametrised by $B$, 
is a family of complex structures $J=\{J_{b}\}_{b\in B}$ on $X$, which are locally of class $C^{(k,\alpha)}$, such that
for every relatively compact domain $\Omega\subset X$ the map $b\mapsto J_{b}|_{\Omega}\in\Gamma^{(k,\alpha)}(\text{End}(TX)|_{\Omega})$
is continuous.
\end{definition}

A continuous family $J=(J_{b})_{b\in B}$ of complex structures on $X$ furnishes us with a Riemann surface $(X,J_{b})$ for every $b\in B$.
A function $f:X\to\CC$ is \textbf{$J_{b}$-holomorphic} if it is holomorphic with respect to the complex structure $J_{b}$ on $X$ and the
standard complex structure on $\CC$.

Let $B$ be a topological space and let $A$ be a subset of $B\times X$. For every $b\in B$ we denote 
\[
A_{b}=\{x\in X\,:\,(b,x)\in A\}.
\]
If $f:A\to\CC$ is a function and $b\in B$, we denote by $f_{b}:A_{b}\to\CC$ the function,
given by
\[
f_{b}(x)=f(b,x)
\]
for $x\in A_{b}$. We are interested in continuous families of holomorphic functions.

\begin{definition}
Let $B$ be a topological space, $k\in\ZZ_{+}$, $0<\alpha<1$ and let $J=(J_{b})_{b\in B}$ be a 
continuous family of complex structures on $X$ of class $C^{(k,\alpha)}$, parametrised by $B$.

(1) Let $U\subset B\times X$ be an open subset. A continuous function $f:U\to\CC$ is \textbf{$J$-holomorphic}, if the function $f_{b}:U_{b}\to\CC$
is $J_{b}$-holomorphic for every $b\in B$. The vector space of all $J$-holomorphic functions on $U$ is denoted by
$\cO_{J}(U)$. We similarly define a $J$-holomorphic map $f:U\to M$, where $M$ is a complex manifold.

(2) Now let $Z\subset B\times X$ be a closed subset. A continuous function $f:Z\to\CC$ is \textbf{$J$-holomorphic} if there exist
an open set $U\subset B\times X$, containing $Z$, and $\tilde{f}\in\cO_{J}(U)$ such that $\tilde{f}|_{Z}=f$. 
The vector space of all $J$-holomorphic functions on $Z$ is denoted by $\cO_{J}(Z)$.
The vector space of all continuous functions $f:Z\to\CC$, for which the function $f_{b}:\Int(Z_{b})\to\CC$ is
$J_{b}$-holomorphic for every $b\in B$, is denoted by $\cA_{J}(Z)$.
\end{definition}

For our construction, we need to consider continuous functions on proper families of compact subsets of $X$, which we recall below.

\begin{definition} 
Let $B$ be a topological space and 
let $\pi:B\times X\to B$ be the projection onto the first factor. A \textbf{family of compact subsets of $X$}, parametrised by $B$, is given by a closed
subset $K\subset B\times X$, for which $K_{b}$ is a compact subset of $X$ for every $b\in B$ (note that $K_{b}$ may be empty). A family of
compact subsets $K$ is \textbf{proper} if the map $\pi|_{K}:K\to B$ is proper, it is \textbf{wide} if $K_{b}$ is non-empty
for every $b\in B$, and it is called \textbf{Runge}  if $K_{b}$ is Runge for every $b\in B$.
\end{definition}

Recall that a continuous map between topological spaces is proper if the preimage of every compact subset is compact, and
that a compact subset $K\subset X$ is Runge if the complement $X\setminus K$ has no relatively compact connected components. A Runge compact set $K$ is 
holomorphically convex in every complex structure on $X$.

As noted in \cite{Forstneric2025}, we have the following characterization of proper families of compact subsets:

\begin{proposition}
Let $B$ be a Hausdorff topological space and let $K\subset B\times X$ be a closed subset. Then $K$ is a proper family of compact subsets of $X$ if and only if 
the following two conditions hold:

(1) For every $b\in B$ the fiber $K_{b}$ is compact,

(2) For every $b_{0}\in B$ and every open subset $U\subset X$ containing $K_{b_{0}}$ there is
a neighbourhood $B_{0}$ of $b_{0}$ in $B$ such that $K_{b}\subset U$ for every $b\in B_{0}$. 
\end{proposition}

Let us now take a look at some examples.

\begin{example}\rm

(1) For every compact subset $K_{0}\subset X$ we have the constant family $K=B\times K_{0}$ of compact subsets of $X$ 
for which $K_{b}=K_{0}$ for every $b\in B$. More generally, 
let $K\subset B\times X$ be a closed subset for which $\pi|_{K}:K\to B$ is a fiber bundle with a compact fiber. Then $K$
is a proper family of compact subsets of $X$.

(2) A proper family of compact subsets of $X$ need not have all fibers homeomorphic. As an example, consider the case when $B=\RR$, $X=\CC=\RR^{2}$ and
denote by $\overline{\DD}\subset\CC$ the closed unit disk. The set
\[
K=\left((-\infty,0]\times\overline{\DD}\right)\cup\left([0,\infty)\times\{0\}\right)
\]
is then a proper family of compact subsets of $X$. On the other hand, let us define the set
\[
\tilde{K}=\left((-\infty,0]\times\{0\}\right)\cup\left\{(x,\tfrac{1}{x})\,|\,x\in(0,\infty)\right\}.
\]
The set $\tilde{K}$ defines a family of compact subsets of $X$ which is not a proper family.
\end{example}

To show that a given set is a proper family of compact subsets of $X$ we easily obtain the following useful criteria.

\begin{proposition}\label{Proposition Proper families}
Let $B$ be a topological space.
\begin{enumerate}
\item Let $K\subset B\times X$ be a proper family of compact subsets of $X$ and let $K'$ be a closed subset of $K$. Then $K'$ is
a proper family of compact subsets of $X$ as well.

\item Let $K_{1},K_{2}\subset B\times X$ be proper families of compact subsets of $X$. Then $K_{1}\cup K_{2}$ is
a proper family of compact subsets as well.
\end{enumerate}
\end{proposition}

Let $K\subset B\times X$ be a wide, proper family of compact subsets of $X$ and
let $\eta:K\to(0,\infty)$ be a positive continuous function. Since $K_{b}$ is non-empty for every $b\in B$, there exists the minimum
\[
\min(\eta)(b)=\min\{\eta(b,x)\,:\,x\in K_{b}\}>0.
\]
We thus obtain a function $\min(\eta):B\to(0,\infty)$, which is continuous if $K$ is a constant or a locally trivial family. In general, however,
the function $\min(\eta)$ is only lower semicontinuous, but we can always find a continuous minorant $m(\eta):B\to(0,\infty)$ of $\min(\eta)$:

\begin{proposition}\label{Proposition minimal function}
Let $B$ be a metrisable space, $K\subset B\times X$ a wide, proper family of compact subsets of $X$ 
and let $\eta:K\to(0,\infty)$ be a continuous function.
\begin{enumerate}
\item The function $\min(\eta):B\to(0,\infty)$ is lower semicontinuous.

\item There exists a continuous function $m(\eta):B\to(0,\infty)$, such that $m(\eta)(b)<\min(\eta)(b)$ holds for every $b\in B$.
If $B$ is a smooth manifold, we can in addition ensure that the function $m$ is smooth.
\end{enumerate}
\end{proposition}
\begin{proof} 
(a) Let $\epsilon>0$ and $b_{0}\in B$. We have to prove that 
there exists an open neighbourhood $U_{b_{0}}$ of $b_{0}$ in $B$ such that $\min(\eta)(b)>\min(\eta)(b_{0})-\epsilon$ for every $b\in U_{b_{0}}$.
Suppose, on the contrary, that such a neighbourhood does not exist for some $b_{0}$. Then there exists a sequence $(b_{n},x_{n})$
of points in $K$ such that $\eta(b_{n},x_{n})\leq \min(\eta)(b_{0})-\epsilon$ for every $n\in\NN$ and $\lim\limits_{n\to\infty}b_{n}=b_{0}$. The set 
$L=\{b_{n}\,|\,n\in\ZZ_{+}\}$ is a compact subset of $B$, hence $(\pi|_{K})^{-1}(L)$ is a compact subset of $K$. We may therefore 
assume, that the sequence $(b_{n},x_{n})$ is convergent with limit $(b_{0},x_{0})\in K$. But then we have
\[
\min(\eta)(b_{0})\leq \eta(b_{0},x_{0})\leq\min(\eta)(b_{0})-\epsilon,
\]
which leads us to a contradiction.

(b) We have shown that for every $b_{0}\in B$ we can find a neighbourhood $U_{b_{0}}$ of $b_{0}$ in $B$ and a number
$\epsilon_{b_{0}}>0$ such that $\epsilon_{b_{0}}<\min(\eta)(b)$ for every $b\in U_{b_{0}}$. Since $B$ is paracompact,
we can find a subset $B'\subset B$ and for every $b\in B'$ an open subset $V_{b}\subset U_{b}$ such
that $\{V_{b}\}_{b\in B'}$ is a locally finite open cover of $B$. Choose a continuous partition of unity $\{\rho_{b}\}_{b\in B'}$,
subordinated to the cover $\{V_{b}\}_{b\in B'}$. The function $m(\eta):B\to(0,\infty)$, defined by
\[
m(\eta)=\sum_{b\in B'}\epsilon_{b}\rho_{b}
\]
is then continuous and satisfies $0<m(\eta)(b)<\min(\eta)(b)$ for every $b\in B$.
\end{proof}

Classical
versions of Runge and Mergelyan approximation theorems show us that we can approximate a holomorphic function on a Runge compact set $K$ arbitrarily
closely on $K$ with global holomorphic functions on $X$. 
In the construction of families of proper holomorphic maps,
we use the following Mergelyan theorem for proper families of compact Runge sets (see Corollary $5.2$ and Remark $5.7$
in \cite{Forstneric2025}).

\begin{theorem}[Mergelyan theorem for proper families of Runge compacts] \label{Theorem Mergelyan}
Let $B$ be a paracompact Hausdorff space, $k\in\ZZ_{+}$, $0<\alpha<1$ and let $J=\{J_{b}\}_{b\in B}$
be a continuous family of complex structures on $X$ of class $C^{(k,\alpha)}$, parametrised by $B$.
Let $K\subset B\times X$ be a proper family of Runge compacts in $X$ and let $\epsilon:B\to(0,\infty)$ be a continuous function.
Then for every $f\in \cA_{J}(K)$ there exists a function $F\in\cO_{J}(B\times X)$ such that 
$\|F_{b}-f_{b}\|_{K_{b}}<\epsilon(b)$ for every $b\in B$.
\end{theorem}

In our construction, we need the following combination of the Mergelyan theorem and Proposition \ref{Proposition minimal function}.

\begin{proposition}\label{Proposition Mergelyan approximation with bounds}
Let $B$ be a metrisable space, $k\in\ZZ_{+}$, $0<\alpha<1$ and let $J=\{J_{b}\}_{b\in B}$
be a continuous family of complex structures on $X$ of class $C^{(k,\alpha)}$, parametrised by $B$. Let $n\in\NN$
and let $K_{1},K_{2},\ldots,K_{n}\subset B\times X$ be wide, proper families of compact subsets of $X$ such that
their union is contained in a proper family $K$ of Runge compacts in $X$. Suppose $f\in \cA_{J}(K)$
is a function that satisfies conditions $\Re f>C_{i}$ on $K_{i}$ for some
positive constants $C_{i}$ for $1\leq i\leq n$. Then for every continuous function $\epsilon:B\to(0,\infty)$ 
there exists a function $F\in\cO_{J}(B\times X)$ such that $\Re F>C_{i}$ on $K_{i}$ for $1\leq i\leq n$
and $\|F_{b}-f_{b}\|_{K_{b}}<\epsilon(b)$ for every $b\in B$.
\end{proposition}
\begin{proof} 
For $1\leq i\leq n$ the function $\Re f-C_{i}$ is continuous and positive on $K_{i}$. By Proposition \ref{Proposition minimal function}, there
exist continuous functions $\delta_{i}:B\to(0,\infty)$ such that for every $(b,x)\in K_{i}$ we have $\Re f(b,x)-\delta_{i}(b)>C_{i}$.
Now define a continuous function $\delta=\min\{\delta_{1},\delta_{2},\ldots,\delta_{n},\epsilon\}:B\to(0,\infty)$. 
From Mergelyan's theorem, it follows that there exists a function
$F\in\cO_{J}(B\times X)$ such that $\|F_{b}-f_{b}\|_{K_{b}}<\delta(b)$ for every $b\in B$. This function $F$ satisfies the conditions.
\end{proof}

\section{Construction of families of proper holomorphic maps}\label{sec:proof}

We first recall how we can reduce the construction of a family of proper holomorphic maps to the construction of a converging
sequence on an exhausting family of compact sets in $X$.

\begin{proposition}\label{Proposition sequence of functions}
Let $B$ be a topological space, $k\in\ZZ_{+}$, $0<\alpha<1$ and let $J=\{J_{b}\}_{b\in B}$
be a continuous family of complex structures on $X$ of class $C^{(k,\alpha)}$, parametrised by $B$.
Let $\emptyset=K_{0}\subset K_{1}\subset K_{2}\subset K_{3}\subset\ldots$ be an exhaustion of 
$X$ by compact sets such that $K_{n}\subset\Int K_{n+1}$ for every $n\in\NN$. 
Suppose that for every $n\in\ZZ_{+}$ we have functions $F_{n,1},F_{n,2}\in \cA_{J}(B\times K_{n})$, such that for every $n\in\NN$ it holds:
\begin{enumerate}
\item [$(a)_{n}$] $|F_{n,i}(b,x)-F_{n-1,i}(b,x)|<\frac{1}{2^{n-1}}$ for every $(b,x)\in{B\times K_{n-1}}$ and $i=1,2$,

\item [$(b)_{n}$] $\max\{\Re F_{n,1}(b,x),\Re F_{n,2}(b,x)\}>n-1$ for every $(b,x)\in B\times(K_{n}\setminus\Int K_{n-1})$.
\end{enumerate}
Then there exist functions $F_{1},F_{2}\in\cO_{J}(B\times X)$ such that $F=(F_{1},F_{2}):B\times X\to\CC^{2}$ is
a continuous $J$-holomorphic map, for which $F_{b}:X\to\CC^{2}$ is a
proper map for every $b\in B$.
\end{proposition}
\begin{proof} 
It follows from the condition $(a)_{n}$ that the sequences $(F_{n,1})_{n\in\NN}$ and $(F_{n,2})_{n\in\NN}$ converge uniformly on
the sets of the form $B\times K$, where $K\subset X$ is a compact subset. 
For the limit functions $F_{1}=\lim\limits_{n\to\infty}F_{n,1}$
and $F_{2}=\lim\limits_{n\to\infty}F_{n,2}$, we have that $F_{1},F_{2}\in\cO_{J}(B\times X)$.

For $n>1$ and $i=1,2$ it then follows from $(a)_{n}$ that 
\[
|F_{i}(b,x)-F_{n,i}(b,x)|<\frac{1}{2^{n}}+\frac{1}{2^{n+1}}+\ldots=\frac{1}{2^{n-1}}<1 \text{ for } (b,x)\in {B\times K_{n}}
\]
and further from $(b)_{n}$ that
\[
\max\{\Re F_{1}(b,x),\Re F_{2}(b,x)\}>n-2 \text{ for } (b,x)\in B\times(K_{n}\setminus\Int K_{n-1}),
\]
which implies that $F_{b}:X\to\CC^{2}$  is a proper map for every $b\in B$.
\end{proof}

We now consider the case $X=\R^2$, where the topology is trivial:

\begin{theorem}\label{Theorem proper maps on C} Let $B$ be a metrisable topological space, 
$k\in\ZZ_{+}$, $0<\alpha<1$ and let $J=\{J_{b}\}_{b\in B}$ be a continuous family of complex structures on $\RR^{2}$ of class $C^{(k,\alpha)}$, 
parametrised by $B$. Then there exists a $J$-holomorphic map $F:B\times\R^2\to\CC^{2}$ for which $F_{b}:\R^2\to\CC^{2}$ is a
proper map for every $b\in B$.
\end{theorem}

To prove Theorem \ref{Theorem proper maps on C} we first introduce some notations, where 
we identify $\R^2$ and $\CC$ for convenience.
We choose the exhaustion of $\CC$
by closed disks
\[
K_{n}=n\overline{\DD}=\{z\in\CC\,:\,|z|\leq n\}
\]
for $n\in\NN$, and denote by
\[
A_{n}=K_{n}\setminus\Int K_{n-1}
\]
the closed annulus in $\CC$ between circles of radii $n-1$ and $n$. Also let $K_{0}=\emptyset$.

\begin{definition}\label{Definition Pictures of sets}
(1) Let $n\in\NN$ and suppose we have $k\in\NN$ angles $0\leq\phi_{1}<\ldots<\phi_{k}<2\pi$. 
We then define the following subsets of $\CC$:
\begin{eqnarray*}
\gamma(n,\{\phi_{1},\ldots,\phi_{k}\})&=&\{re^{i\phi}\,:\,r\in[n,n+1],\phi\in\{\phi_{1},\ldots,\phi_{k}\}\}, \\
p(n,\{\phi_{1},\ldots,\phi_{k}\})&=&\{ne^{i\phi}\,:\,\phi\in\{\phi_{1},\ldots,\phi_{k}\}\}.
\end{eqnarray*}
The set $\gamma(n,\{\phi_{1},\ldots,\phi_{k}\})$ is the union of radial line segments at angles in $\{\phi_{1},\ldots,\phi_{k}\}$, while their
inner endpoints form the set $p(n,\{\phi_{1},\ldots,\phi_{k}\})$.

(2) Let $n\in\NN$ and suppose $\phi_{1},\phi_{2}\in[0,2\pi]$ are such that $0<\phi_{2}-\phi_{1}<2\pi$. We then define:
\begin{eqnarray*}
D(n,\phi_{1},\phi_{2})&=&\{re^{i\phi}\,:\,r\in[n,n+1],\phi\in[\phi_{1},\phi_{2}]\}, \\
\alpha(n,\phi_{1},\phi_{2})&=&\{ne^{i\phi}\,:\,\phi\in[\phi_{1},\phi_{2}]\}.
\end{eqnarray*}
The set $D(n,\phi_{1},\phi_{2})\subset\CC$ is the part of the annulus $A_{n}$ which lies between angles $\phi_{1}$ and $\phi_{2}$.
Its inner boundary arc is denoted by $\alpha(n,\phi_{1},\phi_{2})$.

(3) Let $n\in\NN$, let $\phi_{1},\phi_{2}\in[0,2\pi]$ be such that $0<\phi_{2}-\phi_{1}<2\pi$ and suppose that
$0<\delta<\frac{1}{3}\min\{1,\phi_{2}-\phi_{1}\}$. We then define:
\begin{eqnarray*}
L(n,\delta,\phi_{1},\phi_{2})&=&\{re^{i\phi}\,:\,r\in[n+\delta,n+1],\phi\in[\phi_{1}+\delta,\phi_{2}-\delta]\}, \\
W(n,\delta,\phi_{1},\phi_{2})&=&\overline{D(n,\phi_{1},\phi_{2})\setminus L(n,\delta,\phi_{1},\phi_{2})}.
\end{eqnarray*}
The set $W(n,\delta,\phi_{1},\phi_{2})$ is the closed $\delta$-neighbourhood of $\gamma(n,\{\phi_{1},\phi_{2}\})\cup\alpha(n,\phi_{1},\phi_{2})$
in $D(n,\phi_{1},\phi_{2})$.

(4) Let $n,k\in\NN$ and suppose $0=\phi_{1}<\phi_{2}<\ldots<\phi_{k}<2\pi$. Furthermore, let 
$0<\delta<\frac{1}{3}\min\{1,\phi_{2}-\phi_{1},\phi_{3}-\phi_{2},\ldots,\phi_{k}-\phi_{k-1},2\pi-\phi_{k}\}$.
We then define:
\begin{eqnarray*}
W(n,\delta,\{\phi_{1},\phi_{2},\ldots,\phi_{k}\})&=&W(n,\delta,\phi_{1},\phi_{2})\cup\ldots\cup W(n,\delta,\phi_{k-1},\phi_{k})\cup W(n,\delta,\phi_{k},2\pi), \\
L(n,\delta,\{\phi_{1},\phi_{2},\ldots,\phi_{k}\})&=&L(n,\delta,\phi_{1},\phi_{2})\cup\ldots\cup L(n,\delta,\phi_{k-1},\phi_{k})\cup L(n,\delta,\phi_{k},2\pi).
\end{eqnarray*}

(5) Let $n\in\NN$ and suppose $0=\phi_{1}<\phi_{2}<\ldots<\phi_{k}<2\pi$, where $k$ is an even number. We then define:
\begin{eqnarray*}
D_{\text{odd}}(n,\{\phi_{1},\phi_{2},\ldots,\phi_{k}\})&=&D(n,\phi_{1},\phi_{2})\cup D(n,\phi_{3},\phi_{4})\cup\ldots\cup D(n,\phi_{k-1},\phi_{k}), \\
D_{\text{even}}(n,\{\phi_{1},\phi_{2},\ldots,\phi_{k}\})&=&D(n,\phi_{2},\phi_{3})\cup D(n,\phi_{4},\phi_{5})\cup\ldots\cup D(n,\phi_{k},2\pi).
\end{eqnarray*}
In a similar fashion we define the sets $\alpha_{\text{odd}}$, 
$\alpha_{\text{even}}$, $L_{\text{odd}}$, $L_{\text{even}}$, $W_{\text{odd}}$
and $W_{\text{even}}$.
\end{definition}
All of the above sets are compact subsets of $\CC$ as shown in Figure \ref{Figure Sets}.
\begin{figure}[H]
  \centering
  \includegraphics[width=8cm]{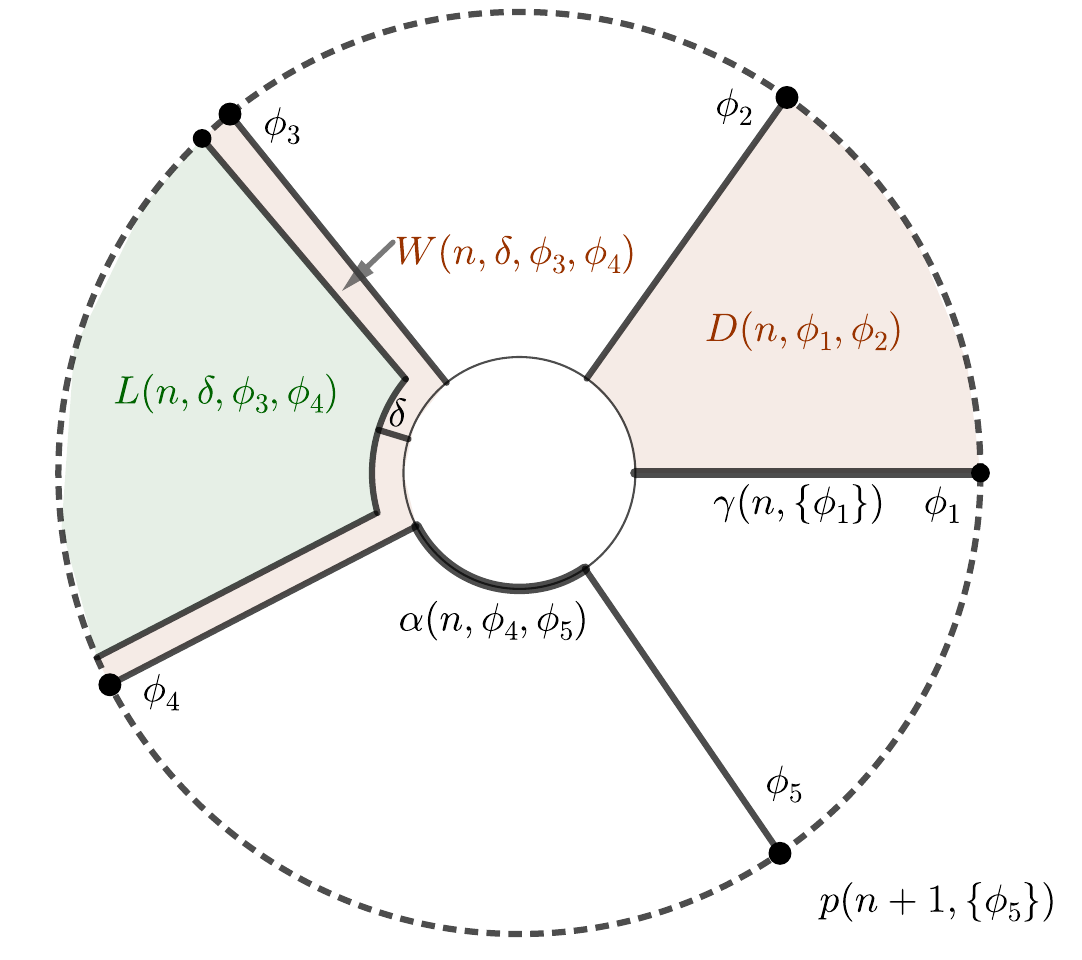}
\caption{Pictures of sets from Definition \ref{Definition Pictures of sets}}\label{Figure Sets}
\end{figure} 
 
Next we extend the above definitions to the setting of $B\times\CC$, where
$B$ is a topological space. If $\phi_{i}:B\to[0,2\pi]$ and $\delta:B\to(0,\frac{1}{3})$ are continuous functions 
that satisfy the conditions $(1)-(5)$ in Definition \ref{Definition Pictures of sets} pointwise, we can define
families of compact subsets of $\CC$ whose fibres are the corresponding sets. We denote such a family by adding a subscript $B$. For example, 
if $n\in\NN$ and $\phi_{1},\phi_{2}:B\to[0,2\pi]$ are continuous functions such that $0<\phi_{2}(b)-\phi_{1}(b)<2\pi$ for every $b\in B$, 
then $D_{B}(n,\phi_{1},\phi_{2})$ is a subset of $B\times\CC$, which is implicitly defined by
\[
(D_{B}(n,\phi_{1},\phi_{2}))_{b}=D(n,\phi_{1}(b),\phi_{2}(b))\subset\CC
\]
for every $b\in B$.

\begin{proposition}
Let $B$ be a topological space. The sets $\gamma_{B}, p_{B}, D_{B},\alpha_{B},L_{B}$ and $W_{B}$ are all proper families of compact subsets
of $\CC$. The same is true for their odd and even versions.
\end{proposition}
\begin{proof}
All these sets are subsets of the constant family $B\times K_{n+1}$, so by Proposition \ref{Proposition Proper families} it suffices to prove they
are closed subsets. This follows from the fact that their complements in $B\times\CC$ are open since the functions $\phi_{i}$ and $\delta$
are continuous. 
\end{proof}

To be able to use the Mergelyan theorem to construct the sequence of functions from Proposition \ref{Proposition sequence of functions}, 
we need the following result.

\begin{proposition}\label{Proposition Collar neighbourhood}
Let $B$ be a metrisable topological space, $n\in\NN$ and let 
$\phi_{1},\phi_{2}:B\to[0,2\pi]$ be continuous functions such that $0<\phi_{2}(b)-\phi_{1}(b)<2\pi$
for every $b\in B$. Suppose $f:D_{B}(n,\phi_{1},\phi_{2})\to\CC$ is a continuous function such that:
\begin{enumerate}
\item $\Re f>n$ on $\gamma_{B}(n,\{\phi_{1},\phi_{2}\})\cup\alpha_{B}(n,\phi_{1},\phi_{2})$,

\item $\Re f>n+1$ on $p_{B}(n+1,\{\phi_{1},\phi_{2}\})$.
\end{enumerate}
Then there exists a continuous function $\delta:B\to(0,\frac{1}{3})$, satisfying 
$\delta<\tfrac{1}{3}\min\{1,\phi_{2}-\phi_{1}\}$ and such that:
\begin{enumerate}
\item $\Re f>n$ on $W_{B}(n,\delta,\phi_{1},\phi_{2})$,

\item $\Re f>n+1$ on $\alpha_{B}(n+1,\phi_{1},\phi_{1}+\delta)\cup\alpha_{B}(n+1,\phi_{2}-\delta,\phi_{2})$.
\end{enumerate}
\end{proposition}
\begin{proof} 
Let us define 
\[
K'=\{(b,x)\in D_{B}(n,\phi_{1},\phi_{2}):\Re f(b,x)\leq n\}\cup\{(b,x)\in\alpha_{B}(n+1,\phi_{1},\phi_{2}):\Re f(b,x)\leq n+1)\}.
\]
The set $K'$ is a closed subset of the proper family of compact subsets $D_{B}(n,\phi_{1},\phi_{2})$, hence $K'$ is a proper family as well. Since $K'$
may have empty fibers, we enlarge it to a proper family of compact subsets
\[
K=K'\cup p_{B}(n+1,\{\tfrac{1}{2}(\phi_{1}+\phi_{2})\}),
\]
for which $K_{b}\neq\emptyset$ for every $b\in B$. Then $K$ is a wide, proper family of compact subsets of $\CC$ which 
is disjoint from $\gamma_{B}(n,\{\phi_{1},\phi_{2}\})\cup\alpha_{B}(n,\phi_{1},\phi_{2})$. 

For any $(b,x)\in K$ we can write $x$ in the form $x=re^{i\phi}$ for unique $r\in(n,n+1]$ and $\phi\in(\phi_{1}(b),\phi_{2}(b))$. 
The functions $\phi_{2}-\phi$, $\phi-\phi_{1}$ and $r-n$ are continuous and positive on $K$. By Proposition \ref{Proposition minimal function}, we can find a function $\delta:B\to(0,\infty)$ such that for every $(b,re^{i\phi})\in K$ we have:
\begin{eqnarray*}
r&\in&(n+\delta(b),n+1], \\
\phi&\in&(\phi_{1}(b)+\delta(b),\phi_{2}(b)-\delta(b)).
\end{eqnarray*}
If needed, we can make $\delta$ smaller, so that $\delta<\tfrac{1}{3}\min\{1,\phi_{2}-\phi_{1}\}$.
We then have $\Re f>n$ on $W_{B}(n,\delta,\phi_{1},\phi_{2})$ and 
$\Re f>n+1$ on $\alpha_{B}(n+1,\phi_{1},\phi_{1}+\delta)\cup\alpha_{B}(n+1,\phi_{2}-\delta,\phi_{2})$.
\end{proof}

\begin{proof}[Proof of Theorem \ref{Theorem proper maps on C}]
Let $l_{n}=3^{n-1}$ for $n\in\NN$.
According to Proposition \ref{Proposition sequence of functions}, it suffices to construct a sequence of functions $F_{n,1},F_{n,2}\in \cA_{J}(B\times K_{n})$ 
that for every $n\in\NN$ satisfy the conditions:
\begin{enumerate}
\item [$(a)_{n}$] $|F_{n,i}(b,x)-F_{n-1,i}(b,x)|<\frac{1}{2^{n-1}}$ for every $(b,x)\in {B\times K_{n-1}}$ and $i=1,2$,
\item [$(b)_{n}$] $\max\{\Re F_{n,1}(b,x),\Re F_{n,2}(b,x)\}>n-1$ for every $(b,x)\in B\times A_{n}$.
\end{enumerate}
We construct such a sequence inductively, together with the sequence of continuous families of angles:  These angles are defined
by continuous functions $\phi_{n,j}:B\to [0,2\pi]$ for $j\in\{1,\ldots,2l_{n}+1\}$, which 
satisfy for every $n\in\N$ the following conditions:
\begin{enumerate}
\item [$(c)_{n}$] $0=\phi_{n,1}(b)<\phi_{n,2}(b)<\cdots<\phi_{n,2l_{n}}(b)<2\pi=\phi_{n,2l_{n}+1}(b)$ for every $b\in B$,
\item [$(d)_{n}$] $\Re F_{n,1}>n$ on $(\alpha_{\text{odd}})_{B}(n,\{\phi_{n,1},\phi_{n,2},\ldots,\phi_{n,2l_{n}}\})$, \\ 
$\Re F_{n,2}>n$ on $(\alpha_{\text{even}})_{B}(n,\{\phi_{n,1},\phi_{n,2},\ldots,\phi_{n,2l_{n}}\})$.
\end{enumerate}

To start with the induction, we define constant functions $\phi_{1,1},\phi_{1,2},\phi_{1,3}:B\to[0,2\pi]$ by 
\[
\phi_{1,1}=0,\,\phi_{1,2}=\pi,\,\phi_{1,3}=2\pi
\]
and choose any functions $F_{0,1},F_{0,2}\in \cA_{J}(B\times K_{0})$ and $F_{1,1},F_{1,2}\in \cA_{J}(B\times K_{1})$ that satisfy the conditions $(a)_1$, $(b)_{1}$ and $(d)_{1}$.
(For example, we could just choose appropriate constant functions $F_{0,1}$, $F_{0,2}$, $F_{1,1}$ and $F_{1,2}$.) 

Suppose that for some $n\in\NN$ we have functions $F_{m,1},F_{m,2}\in \cA_{J}(B\times K_{m})$ and $\phi_{m,j}:B\to[0,2\pi]$  for $j\in\{1,2,\ldots,2l_{m}+1\}$
which satisfy conditions $(a)_{m}$, $(b)_{m}$, $(c)_{m}$ and $(d)_{m}$ for $m\in\{1,2,\ldots,n\}$.
In the induction step, we construct continuous functions 
$$\phi_{n+1,1},\phi_{n+1,2},\ldots,\phi_{n+1,2l_{n+1}+1}:B\to[0,2\pi],$$ that satisfy
\[
0=\phi_{n+1,1}<\phi_{n+1,2}<\ldots<\phi_{n+1,2l_{n+1}}<2\pi=\phi_{n+1,2l_{n+1}+1}
\]
and functions $F_{n+1,1},F_{n+1,2}\in \cA_{J}(B\times K_{n+1})$ that satisfy $(a)_{n+1}$, $(b)_{n+1}$ and $(d)_{n+1}$.
Before we turn to details let us quickly describe the main idea of the induction step. We need to construct functions 
$F_{n+1,1},F_{n+1,2}\in \cA_{J}(B\times K_{n+1})$ for which $\max\{\Re F_{n+1,1},\Re F_{n+1,2}\}>n$ on $B\times A_{n+1}$. To do that,
we split the inductive step into three parts. In the first part, we use Mergelyan's theorem to construct
functions $\tilde{F}_{n,1},\tilde{F}_{n,2}\in\cO_{J}(B\times\RR^{2})$ which satisfy 
$\max\{\Re \tilde{F}_{n,1},\Re \tilde{F}_{n,2}\}>n$ on the subset $\gamma_{B}(n,\{\phi_{n,1},\phi_{n,2},\ldots,\phi_{n,2l_{n}}\})$
of $B\times A_{n+1}$. Next we use Proposition \ref{Proposition Collar neighbourhood} to show that there exists a 
continuous function $\delta:B\to(0,\frac{1}{3})$ such that 
$\max\{\Re \tilde{F}_{n,1},\Re \tilde{F}_{n,2}\}>n$ on the subset $W_{B}(n,\{\phi_{n,1},\phi_{n,2},\ldots,\phi_{n,2l_{n}}\})$
of $B\times A_{n+1}$. In the third part, we use the idea from the proofs in \cite{ALJDG,AlarconForstneric2013IM}
to obtain functions $F_{n+1,1},F_{n+1,2}\in \cA_{J}(B\times K_{n+1})$ for which $\max\{\Re F_{n+1,1},\Re F_{n+1,2}\}>n$ on $B\times A_{n+1}$. 

Let us now describe the details. 
First we construct functions $\tilde{F}_{n,1},\tilde{F}_{n,2}\in\cO_{J}(B\times\RR^{2})$ that satisfy:
\begin{enumerate}
\item [$(a^{1})_{n+1}$] $|\tilde{F}_{n,i}(b,x)-F_{n,i}(b,x)|<\frac{1}{2^{n+1}}$ for $(b,x)\in {B\times K_{n}}$ and $i=1,2$,
\item [$(b^{1})_{n+1}$] $\Re \tilde{F}_{n,1}>n$ on 
$\gamma_{B}(n,\{\phi_{n,1},\ldots,\phi_{n,2l_{n}}\})\cup(\alpha_{\text{odd}})_{B}(n,\{\phi_{n,1},\ldots,\phi_{n,2l_{n}}\})$, \\ 
$\Re \tilde{F}_{n,2}>n$ on 
$\gamma_{B}(n,\{\phi_{n,1},\ldots,\phi_{n,2l_{n}}\})\cup(\alpha_{\text{even}})_{B}(n,\{\phi_{n,1},\ldots,\phi_{n,2l_{n}}\})$,
\item [$(d^{1})_{n+1}$] $\Re \tilde{F}_{n,i}>n+1$ on $p_{B}(n+1,\{\phi_{n,1},\ldots,\phi_{n,2l_{n}}\})$ for $i=1,2$.
\end{enumerate}
To do that we first continuously extend the functions $F_{n,1},F_{n,2}$ from the set $B\times K_{n}$ to the set
$(B\times K_{n})\cup\gamma_{B}(n,\{\phi_{n,1},\ldots,\phi_{n,2l_{n}}\})$ so that 
$\Re F_{n,i}>n$ on $\gamma_{B}(n,\{\phi_{n,1},\ldots,\phi_{n,2l_{n}}\})$ and $\Re F_{n,i}>n+1$ on 
$p_{B}(n+1,\{\phi_{n,1},\ldots,\phi_{n,2l_{n}}\})$ for $i=1,2$.
Now note that $\Re F_{n,1}>n$ on the proper family 
$\gamma_{B}(n,\{\phi_{n,1},\ldots,\phi_{n,2l_{n}}\})\cup(\alpha_{\text{odd}})_{B}(n,\{\phi_{n,1},\ldots,\phi_{n,2l_{n}}\})$
and that $\Re F_{n,1}>n+1$ on the proper family $p_{B}(n+1,\{\phi_{n,1},\ldots,\phi_{n,2l_{n}}\})$
of compact subsets of $\RR^{2}$. 
The union of these two proper families is contained in the proper family $(B\times K_{n})\cup\gamma_{B}(n,\{\phi_{n,1},\ldots,\phi_{n,2l_{n}}\})$
of Runge compacts in $\RR^{2}$, so by 
Proposition \ref{Proposition Mergelyan approximation with bounds} we can find a function $\tilde{F}_{n,1}\in\cO_{J}(B\times\RR^{2})$ that satisfies
$(a^{1})_{n+1}$, $(b^{1})_{n+1}$ and $(d^{1})_{n+1}$. In a similar fashion we also obtain a function $\tilde{F}_{n,2}\in\cO_{J}(B\times\RR^{2})$
which approximates the function $F_{n,2}$.

We now proceed to the second part of the induction step. Consider the function $\tilde{F}_{n,1}$ on the proper family 
$(D_{\text{odd}})_{B}(n,\{\phi_{n,1},\ldots,\phi_{n,2l_{n}}\})$ of compact subsets of $\RR^{2}$. From
Proposition \ref{Proposition Collar neighbourhood} it follows that there exists a continuous function $\delta_{1}:B\to(0,\frac{1}{3})$ such that:  
\begin{enumerate}
\item [$\cdot$] $\Re \tilde{F}_{n,1}>n$ on $(W_{\text{odd}})_{B}(n,\delta_{1},\{\phi_{n,1},\ldots,\phi_{n,2l_{n}}\})$,

\item [$\cdot$] $\Re \tilde{F}_{n,1}>n+1$ on 
$\bigcup\limits_{k=1}^{l_{n}}\left(\alpha_{B}(n+1,\phi_{n,2k-1},\phi_{n,2k-1}+\delta_{1})\cup\alpha_{B}(n+1,\phi_{n,2k}-\delta_{1},\phi_{n,2k})\right)$.
\end{enumerate}
In the sequel, we repeat this argument for the function $\tilde{F}_{n,2}$ on the 
$(D_{\text{even}})_{B}(n,\{\phi_{n,1},\ldots,\phi_{n,2l_{n}}\})$ to obtain a function $\delta_{2}:B\to(0,\frac{1}{3})$ such
that $\tilde{F}_{n,2}$ satisfies conditions:
\begin{enumerate}
\item [$\cdot$] $\Re \tilde{F}_{n,2}>n$ on $(W_{\text{even}})_{B}(n,\delta_{2},\{\phi_{n,1},\ldots,\phi_{n,2l_{n}}\})$,

\item [$\cdot$] $\Re \tilde{F}_{n,2}>n+1$ on 
$\bigcup\limits_{k=1}^{l_{n}}\left(\alpha_{B}(n+1,\phi_{n,2k},\phi_{n,2k}+\delta_{2})\cup\alpha_{B}(n+1,\phi_{n,2k+1}-\delta_{2},\phi_{n,2k+1})\right)$.
\end{enumerate}
Let $\delta=\min\{\delta_{1},\delta_{2}\}$ and define functions
$\phi_{n+1,1},\phi_{n+1,2},\ldots,\phi_{n+1,2l_{n+1}+1}:B\to[0,2\pi]$ by:
\begin{eqnarray*}
\phi_{n+1,3k+1}&=&\phi_{n,k+1}, \\
\phi_{n+1,3k+2}&=&\phi_{n,k+1}+\delta, \\
\phi_{n+1,3k+3}&=&\phi_{n,k+2}-\delta
\end{eqnarray*}
for $k\in\{0,1,\ldots,2l_{n}-1\}$ and $\phi_{n+1,2l_{n+1}+1}=2\pi$.
Observe that these functions satisfy the condition $(c)_{n+1}$ while functions $\tilde{F}_{n,1},\tilde{F}_{n,2}$ satisfy the conditions:
\begin{enumerate}
\item [$(a^{2})_{n+1}$] $|\tilde{F}_{n,i}(b,x)-F_{n,i}(b,x)|<\frac{1}{2^{n+1}}$ for all $(b,x)\in {B\times K_{n}}$ and $i=1,2$,
\item [$(b^{2})_{n+1}$] $\max\{\Re \tilde{F}_{n,1}(b,x),\Re \tilde{F}_{n,2}(b,x)\}>n$ for all $(b,x)\in W_{B}(n,\delta,\{\phi_{n,1},\phi_{n,2},\ldots,\phi_{n,2l_{n}}\})$,
\item [$(d^{2})_{n+1}$] $\Re \tilde{F}_{n,1}>n+1$ on 
$\bigcup\limits_{k=1}^{l_{n}}\left(\alpha_{B}(n+1,\phi_{n,2k-1},\phi_{n,2k-1}+\delta)\cup\alpha_{B}(n+1,\phi_{n,2k}-\delta,\phi_{n,2k})\right)$, \\
$\Re \tilde{F}_{n,2}>n+1$ on 
$\bigcup\limits_{k=1}^{l_{n}}\left(\alpha_{B}(n+1,\phi_{n,2k},\phi_{n,2k}+\delta)\cup\alpha_{B}(n+1,\phi_{n,2k+1}-\delta,\phi_{n,2k+1})\right)$.
\end{enumerate}

In the third part of the induction step, we correct the functions $\tilde{F}_{n,1},\tilde{F}_{n,2}$ so that we obtain the condition
$(b)_{n+1}$ on the set $L_{B}(n,\delta,\{\phi_{n,1},\ldots,\phi_{n,2l_{n}}\})$ as well as the condition $(d)_{n+1}$
on the remaining arcs. Let us define a proper family of Runge compacts by
\[
(A_{\text{odd}})_{n}=(B\times K_{n})\cup (D_{\text{odd}})_{B}(n,\{\phi_{n,1},\ldots,\phi_{n,2l_{n}}\})\cup 
(L_{\text{even}})_{B}(n,\delta,\{\phi_{n,1},\ldots,\phi_{n,2l_{n}}\}),
\]
see the left part of Figure \ref{Figure Annular regions}, and define the function $\overline{F}_{n,1}\in \cA_{J}((A_{\text{odd}})_{n})$ by
\[
\overline{F}_{n,1}(b,x)=
\left\{ \begin{array} {ll}
      \tilde{F}_{n,1}  &;(b,x)\in (B\times K_{n})\cup (D_{\text{odd}})_{B}(n,\{\phi_{n,1},\ldots,\phi_{n,2l_{n}}\}),\\
      n+2   &;\hspace{1mm}(b,x)\in(L_{\text{even}})_{B}(n,\delta,\{\phi_{n,1},\ldots,\phi_{n,2l_{n}}\}).
       \end{array}\right.
\]
\begin{figure}[h]
  \centering
  \includegraphics[width=7.5cm]{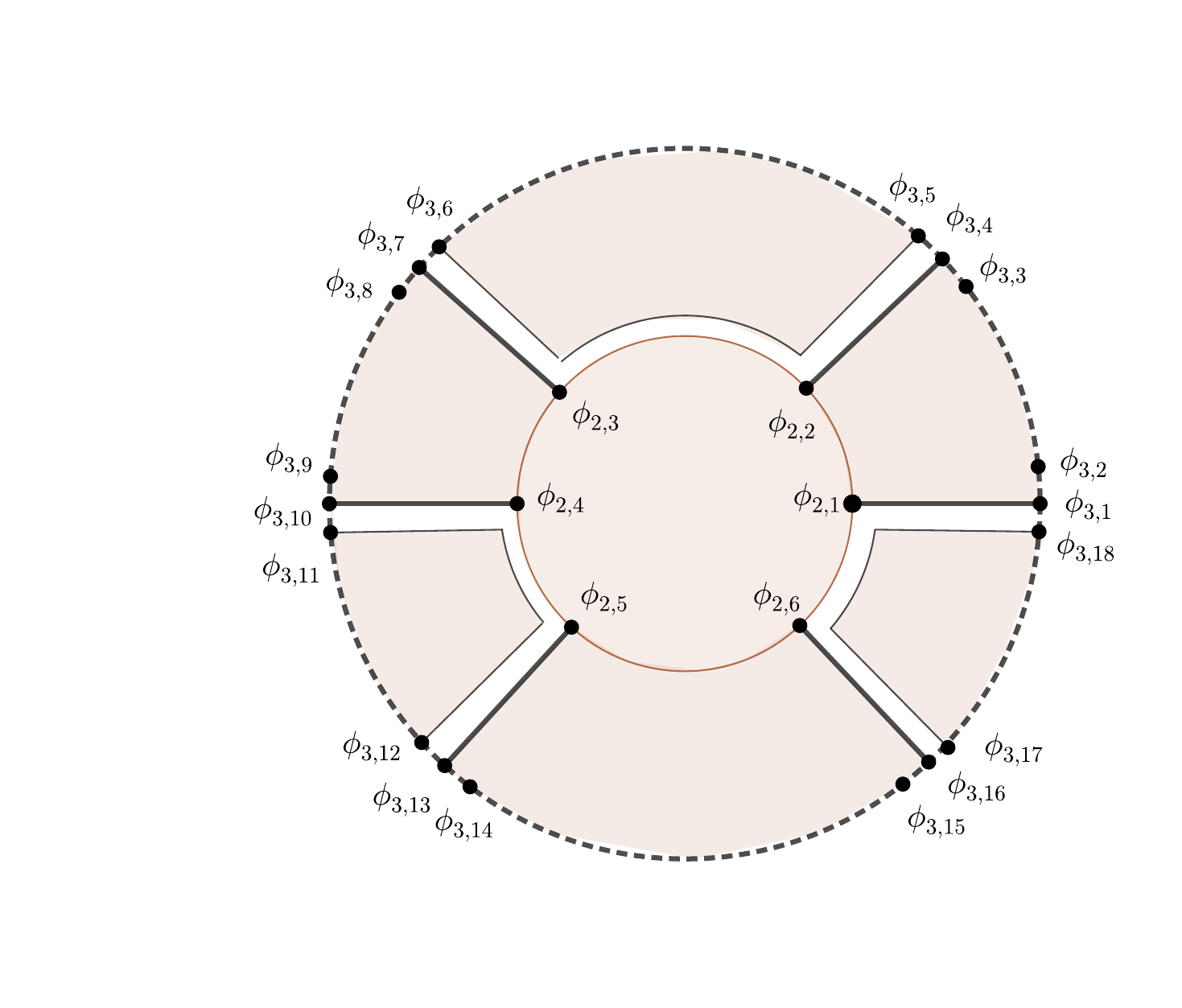}
\hfil
   \includegraphics[width=7.6cm]{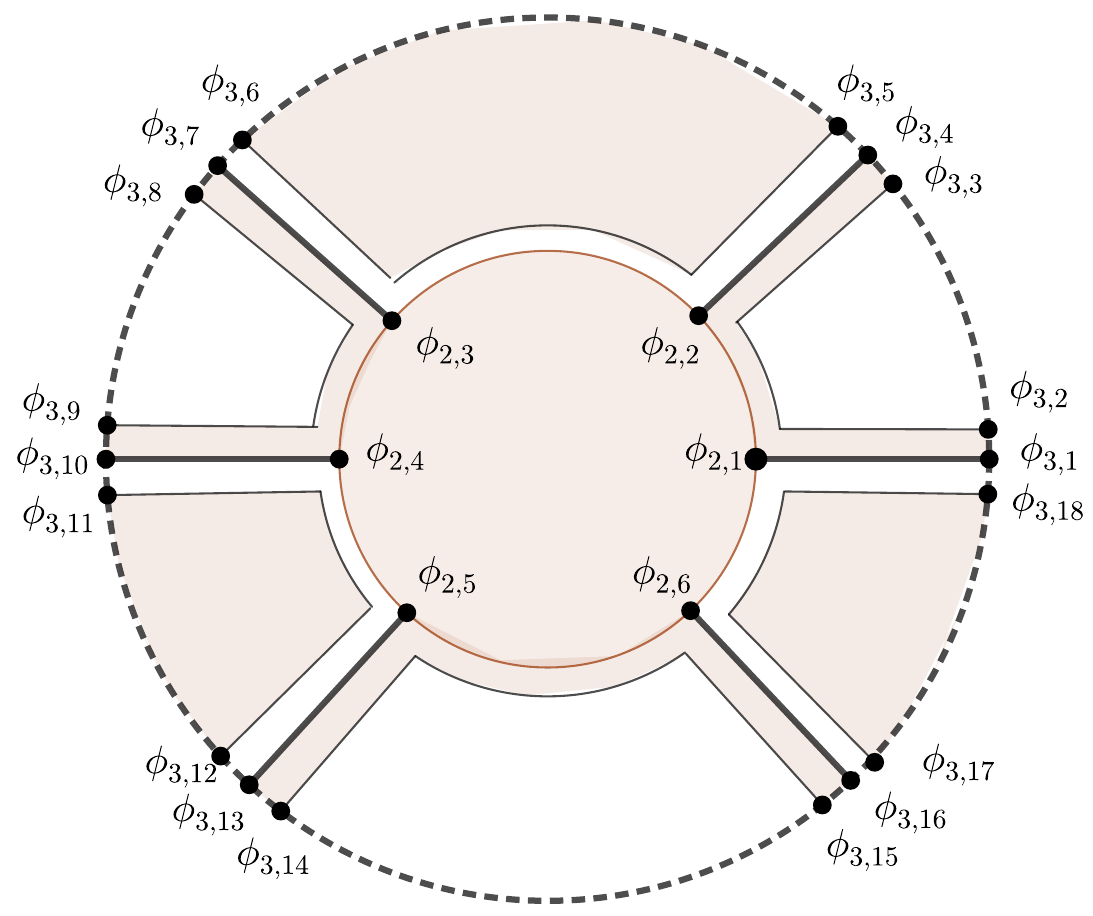}
\caption{Regions in the inductive step in the case $n=2$}\label{Figure Annular regions}
\end{figure} 

Function $\overline{F}_{n,1}$ satisfies conditions:
\begin{enumerate}
\item [$\cdot$] $|\overline{F}_{n,1}(b,x)-F_{n,1}(b,x)|<\frac{1}{2^{n+1}}$ for $(b,x)\in {B\times K_{n}}$,
\item [$\cdot$] $\Re \overline{F}_{n,1}>n$ on $(W_{\text{odd}})_{B}(n,\delta,\{\phi_{n,1},\ldots,\phi_{n,2l_{n}}\})
\cup(L_{\text{even}})_{B}(n,\delta,\{\phi_{n,1},\ldots,\phi_{n,2l_{n}}\})$,
\item [$\cdot$] $\Re \overline{F}_{n,1}>n+1$ on $(\alpha_{\text{odd}})_{B}(n+1,\{\phi_{n+1,1},\ldots,\phi_{n+1,2l_{n+1}}\})$.
\end{enumerate}
By applying Proposition \ref{Proposition Mergelyan approximation with bounds} to the function
$\overline{F}_{n,1}$ with precision at least $\frac{1}{2^{n+1}}$ we obtain a function $F_{n+1,1}\in\cO_{J}(B\times\RR^{2})$ that satisfies:
\begin{enumerate}
\item [$(a^{3,1})_{n+1}$] $|F_{n+1,1}(b,x)-F_{n,1}(b,x)|<\frac{1}{2^{n}}$ for $(b,x)\in {B\times K_{n}}$,
\item [$(b^{3,1})_{n+1}$] $\Re F_{n+1,1}>n$ on $(W_{\text{odd}})_{B}(n,\delta,\{\phi_{n,1},\ldots,\phi_{n,2l_{n}}\})
\cup(L_{\text{even}})_{B}(n,\delta,\{\phi_{n,1},\ldots,\phi_{n,2l_{n}}\})$,
\item [$(d^{3,1})_{n+1}$] $\Re F_{n+1,1}>n+1$ on $(\alpha_{\text{odd}})_{B}(n+1,\{\phi_{n+1,1},\ldots,\phi_{n+1,2l_{n+1}}\})$.
\end{enumerate}
Similarly we obtain a function $F_{n+1,2}\in\cO_{J}(B\times\RR^{2})$ which satisfies conditions:
\begin{enumerate}
\item [$(a^{3,2})_{n+1}$] $|F_{n+1,2}(b,x)-F_{n,2}(b,x)|<\frac{1}{2^{n}}$ for $(b,x)\in {B\times K_{n}}$,
\item [$(b^{3,2})_{n+1}$] $\Re F_{n+1,2}>n$ on $(W_{\text{even}})_{B}(n,\delta,\{\phi_{n,1},\ldots,\phi_{n,2l_{n}}\})
\cup(L_{\text{odd}})_{B}(n,\delta,\{\phi_{n,1},\ldots,\phi_{n,2l_{n}}\})$,
\item [$(d^{3,2})_{n+1}$] $\Re F_{n+1,2}>n+1$ on $(\alpha_{\text{even}})_{B}(n+1,\{\phi_{n+1,1},\ldots,\phi_{n+1,2l_{n+1}}\})$.
\end{enumerate}

The areas in $B\times A_{n+1}$ where $\Re F_{n+1,1}>n$ respectively $\Re F_{n+1,2}>n$ are shown in the right part of Figure \ref{Figure Annular regions}.
Condition $(a)_{n+1}$ now follows from conditions $(a^{3,1})_{n+1}$ and $(a^{3,2})_{n+1}$, condition $(b)_{n+1}$ 
follows from conditions $(b^{3,1})_{n+1}$ and $(b^{3,2})_{n+1}$
while condition $(d)_{n+1}$ follows from conditions $(d^{3,1})_{n+1}$ and $(d^{3,2})_{n+1}$.
The proof is concluded by applying Proposition \ref{Proposition sequence of functions}.
\end{proof}

\begin{proof}[Proof of Theorem \ref{Main theorem}]
Let $K_0=\emptyset$.
Choose a point $b_0\in B$, and a strongly $J_{b_0}$-subharmonic Morse exhaustion function $\tau:X\to (0,\infty)$.
By a small perturbation, we may assume that there is exactly one critical point at every critical level set. 
Choose an increasing sequence $(c_n)_{n\in \N}$ of regular values of $\tau$ converging to $\infty$ such that the interval $(c_n,c_{n+1})$
contains at most one critical value of $\tau$. Then $K_n=\{x\in X\,:\,\tau(x)\le c_n\}$ is a smoothly bounded compact Runge set,
and we may assume that $c_1$ is chosen so large that $K_1$ is nonempty and so small that it is simply connected.
Then $bK_{n}$ is a union of finitely many, say $k_n$, smooth closed Jordan curves. 
If $\tau$ has no critical values in $(c_n,c_{n+1})$, then $K_{n+1}\setminus \Int K_n$ is a union of $k_n$ annular regions, and we call this the 
\emph{noncritical case}. In this case, there is no change in the topology, and the construction is similar to the construction in the proof  
of Theorem \ref{Theorem proper maps on C}. We will explain the details below.
In the \emph{critical case}, $\tau$ has exactly one critical point in $K_{n+1}\setminus \Int K_n$ of index $0$ or $1$. 

If its index is $0$, then it is a minimum of $\tau$ and a new simply connected component appears.
If its index is $1$, then there is a compact Jordan arc $\gamma_n\subset \Int  K_{n+1}\setminus  \Int K_n $ 
transversally attached with both endpoints to $K_n$, and otherwise disjoint from $K_n$, such that $K_n \cup \gamma_n$ 
is a Runge set and a strong deformation retract of $K_{n+1}$. We need to distinguish two cases: either the 
endpoints of the arc  $\gamma_n$ lie on the same component of
$bK_n$ or the arc connects two different components of $bK_n$. 
We choose two distinct points, denoted by $p^j_n$ and $q^j_n$  on each boundary component of $bK_n$ ($j=1,\ldots, k_n$) such that
the endpoints of the arc $\gamma_n$ are $p^j_n$ and $q^l_n$ for some $j,l\in\{1,\ldots, k_n\}$.
The map $F$ is constructed inductively and 
at the critical case, we need to continuously extend the maps $F_{n,1},F_{n,2}:B\times b\gamma_n\to \{z\in\CC:\Re z>n\}$ to maps
$F_{n,1},F_{n,2}:B\times \gamma_n\to \{z\in\CC:\Re z>n\}$. 
This is possible since the set $\{z\in\CC:\Re z>n\}$ is contractible.
Moreover, we will also obtain a continuously varying family of points on each 
boundary component of $K_n$, which corresponds to the continuous family of angles in the proof of Theorem \ref{Theorem proper maps on C},
and the points  $p^j_n$ and $q^j_n$ will correspond to the constant angles with the different parity: for this reason, we choose for each $n$ and for
each $j\in \{1,\ldots k_n\}$ a continuous map $\varphi_n^j$ from $[0,2\pi]$ to the $j$-th component of $bK_n$ which induces a homeomorphism from
the quotient $[0,2\pi]/(0\sim 2\pi)$ to the $j$-th component of $bK_n$, inducing the given orientation. 
Furthermore, we may achieve that $\varphi_{n}^{j}(0)=\varphi_{n}^{j}(2\pi)=p^{j}_{n}$ and $\varphi_{n}^{j}(\pi)=q^{j}_{n}$.

We inductively construct functions $F_{n,1},F_{n,2}\in \cA_{J}(B\times K_{n})$, $n\in\N\cup \{0\}$, 
positive integers $l^j_{n}$, $j\in\{1,\ldots, k_n\}, n\in\N$, continuous functions
$\phi_{n,m}^j:B\to [0,2\pi]$, $m\in\{1, \ldots, 2l^j_{n}+1\}, j\in\{1,\ldots, k_n\}, n\in\N$, that satisfy the following conditions for every $n\in\N$:
\begin{enumerate}
\item [$(a)_n$] $|F_{n,i}(b,x)-F_{n-1,i}(b,x)|<\frac{1}{2^{n-1}}$ for $(b,x)\in {B\times K_{n-1}}$ and $i=1,2$,
\item [$(b)_n$] $\max\{\Re F_{n,1}(b,x),\Re F_{n,2}(b,x)\}>n-1$ for every $(b,x)\in B\times (K_{n}\setminus\Int K_{n-1})$,
\item [$(c)_n$] $ 0=\phi_{n,1}^j(b)<\phi_{n,2}^j(b)<\cdots<\phi_{n,2l^j_{n}}^j(b)<2\pi=\phi_{n,2l^j_{n}+1}^j(b)$ for each $b\in B$ and \goodbreak $ j\in\{1,\ldots, k_n\}$;
 for each $j\in\{1,\ldots, k_n\}$ there is $m^j_{n}\in\{1, \ldots, l^j_{n}\}$ such that $\phi_{n,2m^j_{n}}^j\equiv \pi$,
\item [$(d)_n$] $ \Re F_{n,1}(b,x)>n$ for $x\in \varphi_n^j ([\phi_{n,2m-1}^j(b),\phi_{n,2m}^j(b)])$ 
and\\
$ \Re F_{n,2}(b,x)>n$ for $x\in \varphi_n^j ([\phi_{n,2m}^j(b),\phi_{n,2m+1}^j(b)])$ for each $m\in\{1, \ldots, l^j_{n}\}$, and \goodbreak $j\in\{1,\ldots, k_n\}$.
\end{enumerate}

Once we complete the construction, the proof is complete due to Proposition \ref{Proposition sequence of functions}.

To start the induction take $F_{0,1}=F_{0,2}=F_{1,1}=F_{1,2}=2$, $l_1^j=1$, $m_1^j=1$ for  $j\in\{1,\ldots, k_1\}$, which satisfy $(a)_{1}-(d)_{1}$.

Assume we have already constructed $F_{i,1},F_{i,2}$,  $l^j_{i}$, $m^j_{i}$, 
$\phi_{i,m}^j$, $m\in\{1, \ldots, 2l^j_{i}+1\}$, $j\in\{1,\ldots, k_i\}$, $i\in\{1,\ldots n\}$, that satisfy $(a)_{i}-(d)_{i}$ for all $i\in\{1,\ldots n\}$.
We construct the functions $F_{n+1,1},F_{n+1,2}$ by dividing each component of the set $K_{n+1}\setminus \Int K_{n}$ into two unions of simply
connected regions that play the roles of $(D_{\text{odd}})_{B}$ and $(D_{\text{even}})_{B}$ in the proof of Theorem \ref{Theorem proper maps on C}.

In the noncritical case, the set $K_{n+1}\setminus \Int K_{n}$ is homeomorphic to a disjoint union of $k_{n}=k_{n+1}$ annular components which
we denote by $A_{n}^{j}$ for $j=1,2,\ldots,k_{n}$. Suppose that the boundary components of $A_{n}^{j}$ are parametrised by
$\varphi_{n}^{j_{\text{inner}}}$ and $\varphi_{n+1}^{j_{\text{outer}}}$. We may choose a diffeomorphism 
$\psi^{j}_{n}:\{z\in\CC\,:\,1\le |z|\le 2\} \to A^{j}_{n}$ such that
$\psi^{j}_{n}(\e^{i t}) =\varphi_{n}^{j_{\text{inner}}}(t)$ and $\psi^{j}_{n}(2\e^{i t})=\varphi_{n+1}^{j_{\text{outer}}}(t)$ for $t\in[0,2\pi)$.
Denote by $\gamma'^{j}_{n}$ the arc $\psi^{j}_{n}([1,2])$ and by $\gamma''^{j}_{n}$ the arc $\psi^{j}_{n}([-2,-1])$.
By the property $(d)_n$ we can continuously extend the maps $F_{n,1},F_{n,2}$ from  $B\times K_{n}$ to maps from $B\times (K_{n}\cup\gamma'^j_n\cup \gamma''^j_n)$  
so that the image 
of $B\times (\gamma'^j_n\cup \gamma''^j_n)$  lies in $\{z\in\CC:\Re z>n\}$
and the image of  $B\times \{p_{n+1}^{j}\}$ and of $B\times \{q_{n+1}^{j}\}$ lies in $\{z\in\CC:\Re z>n+1\}$.
Then we proceed as  in the proof of Theorem \ref{Theorem proper maps on C} to obtain 
functions  $F_{n+1,1},F_{n+1,2}\in \cA_{J}(B\times K_{n+1})$, integers $l^j_{n+1}$, $m^j_{n+1}$, 
and functions $\phi_{n+1,i}^j$ ($i\in\{1, \ldots, 2l^j_{n+1}+1\}$, $j\in\{1,\ldots, k_{n+1}\}$) satisfying properties $(a)_{n+1}-(d)_{n+1}$.

In the critical case, we only need to consider critical points with the index $1$, since we can treat the new appearing component in the case of critical points with index $0$ 
in the same way as at the start of the inductive construction.
Thus, we first consider the situation in which the arc $\gamma_n$ connects two different components of $bK_n$. 
Then the number of components of $bK_{n+1}$ is one less than the number of components of $bK_{n}$. By rearranging the notation, we may assume
that $\gamma_n$ connects $p^{k_n-1}_n$ and $q^{k_n}_n$.
The set  $K_{n+1}\setminus\Int K_{n}$ is a union of a two-connected domain $D_n$ in $X$,  and perhaps a
finite number of annuli, where the arc $\gamma_n\subset D_n$ connects two components of the complement of $D_n$ in $X$. 
In the annular regions of $K_{n+1}\setminus\Int K_{n}$  we proceed as in the 
noncritical case, thus we provide the details only for the construction corresponding the  domain $D_n$.
Since the domain $D_n$ is two-connected, we first explain how we choose continuous family of arcs connecting the boundary of $bK_n$ and $bK_{n+1}$,
corresponding to the arcs $\gamma(n,\{\phi_1,\ldots,\phi_k\})$ in the proof of
Theorem \ref{Theorem proper maps on C}. 
We can choose  pairwise disjoint smooth arcs $\gamma^{p}_n$ and $\gamma^{q}_n$ in $K_{n+1}\setminus (\Int K_n \cup \gamma_n)$ 
which intersect $bK_n$ and $bK_{n+1}$ transversally at their endpoints
such that the endpoints of $\gamma^{p}_n$  are $p_n^{k_n}$ and $p_{n+1}^{k_{n+1}}$, 
and the endpoints of $\gamma^{q}_n$ are $q_{n}^{k_n-1}$ and $q_{n+1}^{k_{n+1}}$, see the left part of Figure \ref{Figure Critical case 1}. 
\begin{figure}[h]
  \centering
  \includegraphics[width=7.5cm]{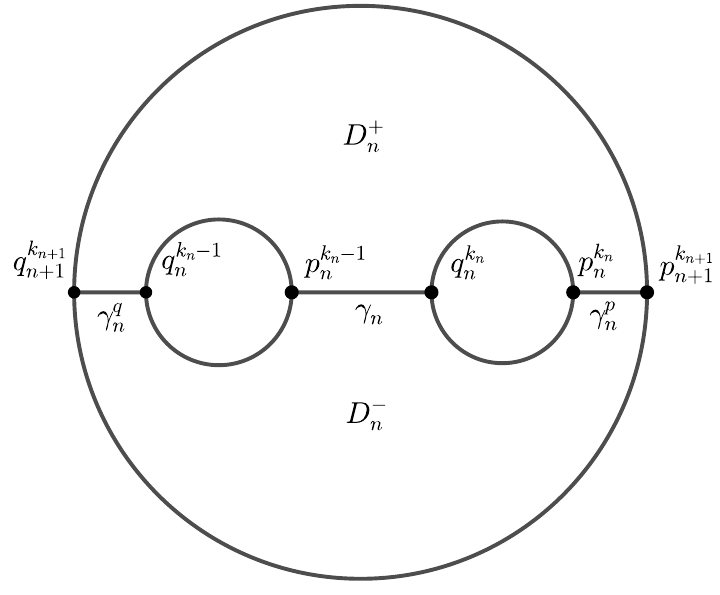}
\hfil
   \includegraphics[width=7.8cm]{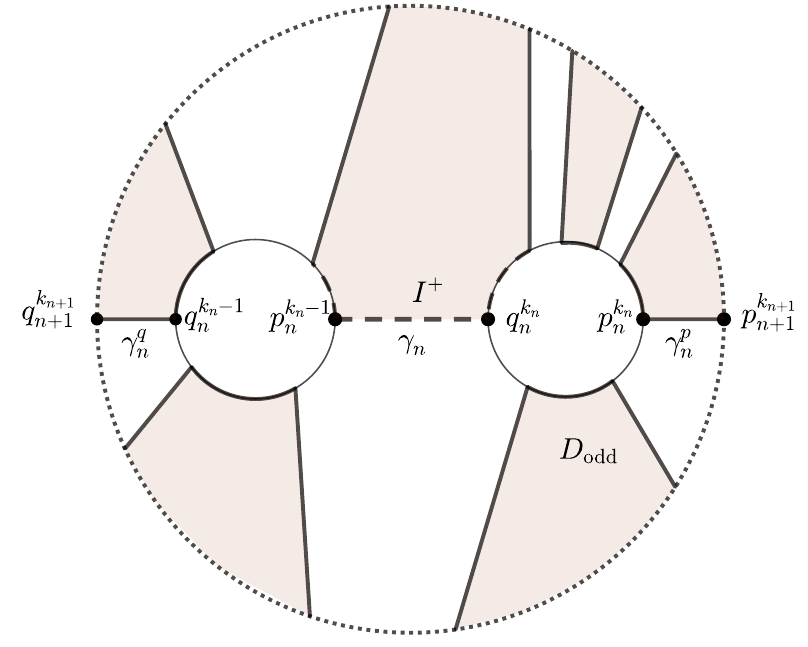}
\caption{Critical case 1}\label{Figure Critical case 1}
\end{figure}  

Then the domain $D_{n}$ is the union of two closed simply connected domains $D_{n}^{+}$ and $D_{n}^{-}$ with the arcs 
$\gamma_{n}$, $\gamma^{p}_{n}$ and $\gamma^{q}_{n}$ as their common boundary.
There is a diffeomorphism $\Psi_{n}^{+}$ from $D_{n}^{+}$ 
to the convex hull $C$ of points $(2,0), (2,1), (1,2), (-1,2), (-2,1), (-2,0)$ in $\RR^{2}$
that maps $q_{n+1}^{k_{n+1}}$ to $(-2,0)$, $q_{n}^{k_{n}-1}$ to $(-2,1)$, $p_{n}^{k_{n}-1}$ 
to $(-1,2)$, $q_{n}^{k_{n}}$ to $(1,2)$, $p_{n}^{k_{n}}$ to $(2,1)$, and $p_{n+1}^{k_{n+1}}$ to $(2,0)$ 
(and similarly for $D_{n}^{-}$).
The vertical segments in $C$ provide arcs in $D_{n}^{+}$. 
More precisely, for any $\phi\in (0,\pi)$ the points $\varphi_{n}^{k_{n-1}}(\phi)$, $\varphi_n^{k_{n}}(\phi)$ from $bK_{n}$ are 
mapped to points $(x',y')$, $(x'',y'')$ for some $x'\in (-2,-1)$, $x''\in (1,2)$ and $y',y''>0$  on the beveled edges of $C$. 
Then segments from $(x',y')$ to $(x',0)$
and $(x'',y'')$ to $(x'',0)$ mapped back to $D_{n}^{+}$ by  $(\Psi_{n}^{+})^{-1}$ give the required arcs. 
This construction gives a Runge family, and the proof is reduced to the proof in the noncritical case:
First we extend the maps $F_{n,1},F_{n,2}$ continuously to maps from 
$B\times (K_n\cup\gamma_{n}\cup \gamma^{p}_{n}\cup \gamma^{q}_{n})$ 
such that the image of $B\times (\gamma_{n}\cup \gamma^{p}_{n}\cup \gamma^{q}_{n})$ 
lies in $\{z\in\CC:\Re z>n\}$, and
that the image  of  $B\times \{p_{n+1}^{k_{n+1}}\}$ and $B\times \{q_{n+1}^{k_{n+1}}\}$ 
lies in $\{z\in\CC:\Re z>n+1\}$. 
The continuous extension to $B\times(\gamma_{n}\cup \gamma^{p}_{n}\cup \gamma^{q}_{n})$ with the image in  $\{z\in\CC:\Re z>n\}$ is possible by the property $(d)_n$ and 
 since the former set is contractible.
Now the construction can proceed similarly to the construction in the regular case, and
here we explain the main differences: 
In the noncritical case, the functions $\phi_{n,j}$ determined  boundary arcs $(\alpha_\text{odd})_B$ 
and $(\alpha_\text{even})_B$ such that $\Re F_{n,1}>n$ on  $(\alpha_\text{odd})_B$, and
$\Re F_{n,2}>n$ on  $(\alpha_\text{even})_B$.
In the critical case, we start at the point $p_n^{k_n}$ on $bK_n$ and move in the positive direction along $k_n$-th component of $bK_n$; 
we first get some arcs with alternating parity until we reach the point $\varphi_n^{k_n}(\phi_{n,2m_n^{k_n}-1}^{k_n}(b))$. 
These arcs determine the domains in $D_n^+$ that belong to 
$D_\text{odd}$, $D_\text{even}$ as before. Observe that for all $b\in B$ and $x\in \varphi_n^{k_n}([\phi_{n,2m_n^{k_n}-1}^{k_n}(b),\pi])\cup \gamma_n\cup 
\varphi_n^{k_n-1}([0, \phi_{n,2}^{k_n-1}(b)])=:I^+(b)$ we have
$\Re F_{n,1}(b,x)>n$. Therefore, the set $I^+$ can be seen as a part of $(\alpha_\text{odd})_B$, and the corresponding domain as a part of $D_\text{odd}$.
As we move further along the boundary of the $(k_n-1)$-th component of $bK_n$ in the positive direction, from $\varphi_n^{k_n-1}(\phi_{n,2}^{k_n-1}(b))$
to $\varphi_n^{k_n-1}(\phi_{n,2l_n^{k_n-1}}^{k_n-1}(b))$ we obtain alternating arcs and domains as before, 
first we get some from $D_n^+$ and then some in $D_n^-$.
Similarly to the above, we have
for all $b\in B$ and $x\in \varphi_n^{k_n-1}([\phi_{n,2l_n^{k_n-1}}^{k_n-1}(b),2\pi])\cup \gamma_n\cup 
\varphi_n^{k_n}([\pi, \phi_{n,2m_n^{k_n}+1}^{k_n}(b)])=:I^-(b)$ that
$\Re F_{n,2}(b,x)>n$, and the set $I^-$ can be viewed as a part of $(\alpha_\text{even})_B$, and the corresponding domain as a part of $D_\text{even}$.
As we move further along the boundary of the $k_n$-th component of $bK_n$, we get  some arcs with alternating parity until we reach the starting point $p_n^{k_n}$.
Then we proceed with the proof  as in the noncritical case. On the right part of Figure \ref{Figure Critical case 1}, we denoted the arcs in 
$(\alpha_\text{odd})_B$ darker than the arcs in $(\alpha_\text{even})_B$.

In the second case, the endpoints of the arc  $\gamma_n$ lie on the same component of
$bK_n$. In this case, the number of components of $bK_{n+1}$ is one greater than the number of components of $bK_{n}$.
By rearranging the notation, we may assume that the endpoints of $\gamma_n$ are $p^{k_n}_n$ and $q^{k_n}_n$. 
The set  $K_{n+1}\setminus\Int K_{n}$ is a union of a two connected domain $D_n$ in $X$,  and perhaps a
finite number of annuli,
where the set $\gamma_n\cup (bK_n\cap D_n)$ separates the other boundary components of $D_n$.
We can choose smooth arcs $\gamma'^{p}_{n}$, $\gamma'^{q}_{n}$, $\gamma''^{p}_{n}$, $\gamma''^{q}_{n}$ in $K_{n+1}\setminus \Int K_{n}$ 
which intersect $bK_{n}$ and $bK_{n+1}$ transversally and only at their endpoints, 
such that the endpoints of $\gamma'^{p}_{n}$ are $p_{n}^{k_{n}}$ and $p_{n+1}^{k_{n+1}-1}$, 
the endpoints of  $\gamma'^{q}_{n}$ are  $q_{n}^{k_{n}}$ and $q_{n+1}^{k_{n+1}-1}$,
the endpoints of $\gamma''^{p}_{n}$ are $p_{n}^{k_{n}}$ and $p_{n+1}^{k_{n+1}}$, the endpoints of $\gamma''^{q}_{n}$ are $q_{n}^{k_{n}}$ and 
$q_{n+1}^{k_{n+1}}$. Furthermore, we can achieve that arcs $\gamma_{n}$, $\gamma'^{p}_{n}$, $\gamma'^{q}_{n}$, $\gamma''^{p}_{n}$, 
$\gamma''^{q}_{n}$ intersect pairwise at most at their endpoints.
We denote by $D_{n}^{+}$ the simply connected component of the set 
$D_{n}\setminus (\gamma'^{p}_{n}\cup \gamma'^{q}_{n})$. Assume that $D_{n}^{+}$ contains $\varphi_{n}^{k_{n}}((0,\pi))$, 
the other case is symmetrical.
Let $D_{n}^{-}$ be the simply connected component of the set $D_{n}\setminus (\gamma''^{p}_{n}\cup \gamma''^{q}_{n})$ 
which contains $\varphi_{n}^{k_{n}}((\pi,2\pi))$. See the left part of Figure \ref{Figure Critical case 2}.

There is a diffeomorphism $\psi_{n}^{+}$ from $\bar D_{n}^{+}$ (and $\psi_{n}^{-}$ from $\bar D_{n}^{-}$) 
to the unit square $[0,1]\times[0,1]$ in $\RR^{2}$  
that maps the arcs 
$\gamma'^{p}_{n}$, $\gamma'^{q}_{n}$ ($\gamma''^{p}_{n}$ ,$\gamma''^{q}_{n}$) to the vertical edges, 
and the arc $\varphi_{n}^{k_{n}}((0,\pi))$, ($\varphi_{n}^{k_{n}}((\pi,2\pi))$) to the upper edge of the square. 
By the properties $(c)_{n}-(d)_{n}$ there are continuous functions $\phi^{p\pm}_{n},\phi^{q\pm}_{n}\colon B\to (0,2\pi)$ such that 
for each $b\in B$ we have $\phi^{p+}_{n}(b)\in (0,\phi_{n,2}^{k_{n}}(b))$, $\phi^{p-}_{n}(b)\in (\phi_{n,2l_n^{k_{n}}}^{k_{n}}(b),2\pi)$,
$\phi^{q+}_{n}(b)\in (\phi_{2m_n^{k_n}-1}^{k_n}(b),\pi)$, $\phi^{q-}_n(b)\in (\pi,\phi^{k_n}_{2m_n^{k_n}+1}(b))$
and such that the restrictions of the maps $F_{n,1},F_{n,2}$ to the arcs $\varphi_{n}^{k_{n}}([0,\phi^{p+}_{n}(b)])$, 
$\varphi_{n}^{k_{n}}([\phi^{p-}_{n}(b),2\pi])$, $\varphi_{n}^{k_{n}}([\phi^{q+}_{n}(b),\pi])$ and $\varphi_{n}^{k_{n}}([\pi,\phi^{q-}_{n}(b)])$
map into $\{z\in\CC:\Re z>n\}$. Note that in this step we added $4$ functions to the family  $\phi^{k_{n}}_{n,i}$.
For every $b\in B$ and every $t\in [0,1]$ we get a segment from $(t,0)$ to 
$((1-t)\psi_n^+(\varphi_n^{k_n}(\phi^{p+}_n(b)))+t\psi_n^+(\varphi_n^{k_n}(\phi^{q+}_n(b))),1)$ in the unit square, 
and by pushing back with 
$(\psi_n^+)^{-1}$ we obtain a family of arcs in $\bar D_n^+$ that correspond to the union of radial line segments (and similarly for 
$\bar D_n^-$).
In particular, for $t=0$ we get arc from $ \varphi_n^{k_n}(\phi^{p+}_n(b))$ to $p_{n+1}^{k_{n+1}-1}$, and 
for $t=1$ we get the arc from $ \varphi_n^{k_n}(\phi^{q+}_n(b))$ to $q_{n+1}^{k_{n+1}-1}$.
Next, we explain how we divide the domain $D_{n}$ into domains $D_\text{odd}$ and $D_\text{even}$ which reduces the proof to the proof 
in the noncritical case (see the right part of Figure \ref{Figure Critical case 2}). 
\begin{figure}[H]
  \centering
  \includegraphics[width=6.5cm]{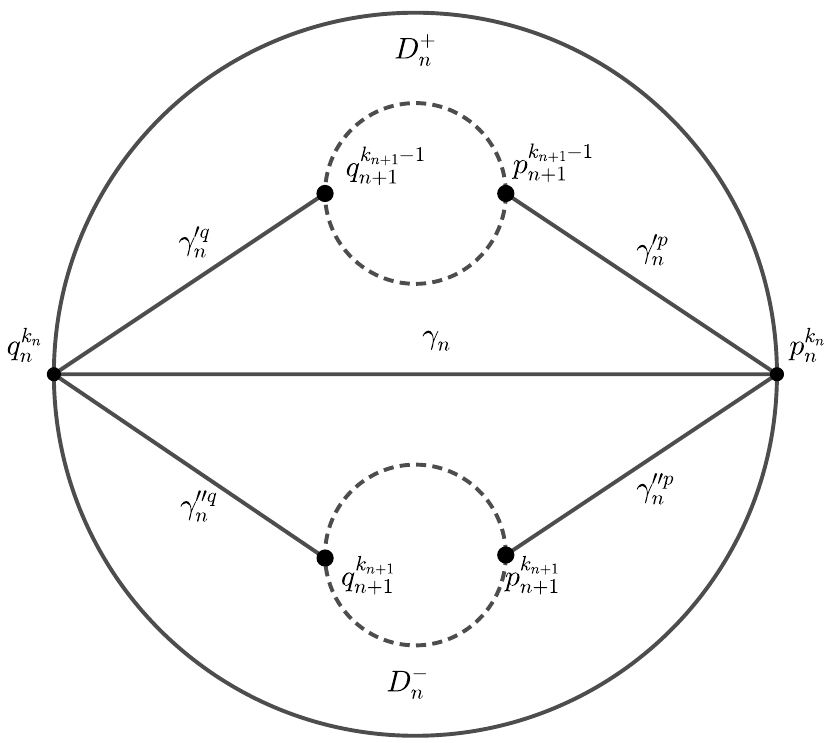}
\hfil
   \includegraphics[width=7.5cm]{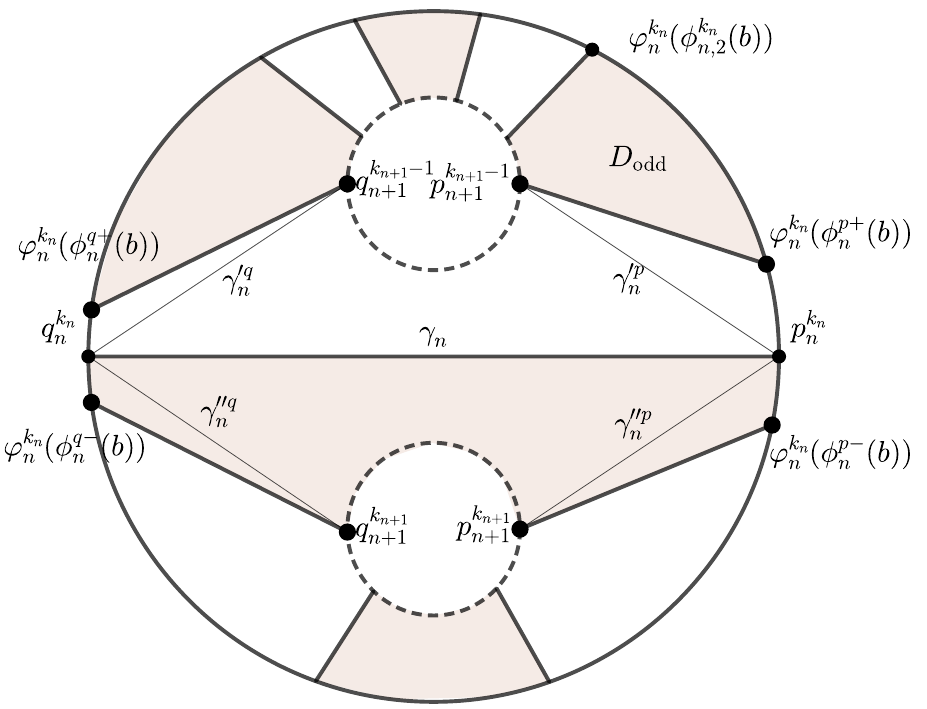}
  \caption{Critical case 2}\label{Figure Critical case 2}
\end{figure}
We start with the point $\varphi_n^{k_n}(\phi_n^{p+}(b))$ and move in the positive
direction along $k_{n}$-th component of $bK_{n}$. We get arcs with alternating parity until 
we reach the point $\varphi_{n}^{k_{n}}(\phi_{n}^{q+}(b))$ determining the domains which alternately belong to $D_\text{odd}$, $D_\text{even}$. 
For $b\in B$ and $x\in\varphi_n^{k_{n}}([\phi_{n}^{q+}(b),\pi])\cup \gamma_n\cup 
\varphi_n^{k_{n}}([0,\phi_{n}^{p+}(b)])=:I^{+}(b)$ it holds that
$\Re F_{n,2}(b,x)>n$, thus, $I^+$ can be taken as a part of $\alpha_\text{even}$ and the corresponding domain as a part of $D_\text{even}$.
For $b\in B$ and $x\in\varphi_n^{k_n}([\pi,\phi_n^{q-}(b)])\cup \gamma_n\cup 
\varphi_n^{k_n}([\phi_n^{p-}(b),2\pi])=:I^{-}(b)$ we have that
$\Re F_{n,1}(b,x)>n$, thus, $I^{-}$ can be taken as a part of $\alpha_\text{odd}$ and the corresponding domain as a part of $D_\text{odd}$.
From the point $\varphi_n^{k_n}(\phi_n^{q-}(b))$ we move in the positive direction along $bK_n$ until we 
reach the point $\varphi_n^{k_n}(\phi_n^{p-}(b))$ and again the points $\varphi_n^{k_n}(\phi_{n,i}^{k_n}(b))$
determine the arcs with alternating parity. Again, this reduces the construction to the noncritical case, which completes the proof.
\end{proof}

\begin{proof}[Proof of Theorem \ref{harmonic}]
In the proof of Theorem \ref{Main theorem}, we constructed a continuous map  
 $F:B\times X\to\CC^{2}$ such that for every $b\in B$ the map $F(b,\cdot):(X,J_b)\to\CC^{2}$ is proper holomorphic,
 and, furthermore, for every $b\in B$, $\max\{ \Re F_1(b,\cdot),\Re F_2(b,\cdot)\}$ goes to infinity as we leave any compact set of $X$, which implies that
the map $(\Re F_1,\Re F_2)(b,\cdot):(X,J_b)\to\RR^2$ is proper harmonic.
\end{proof}

\noindent
{\bf Acknowledgements:} 
The first named author is supported by the European Union (ERC Advanced grant HPDR, 101053085 to F. Forstnerič)
and by the research program P1-0291 from ARIS, Republic of Slovenia.
The second named author is supported by the research program P1-0291 from ARIS, Republic of Slovenia. 
The authors wish to thank the members of the Complex Analysis Seminar in Ljubljana for their remarks, in particular Rafael B. Andrist and Franc Forstneri\v c.


\end{document}